\documentclass[10pt,intlimits,draft]{amsart}
\usepackage{amsfonts,amsmath,amsthm,IMjournal}
\usepackage{enumerate}
\usepackage{mathrsfs}

\def\semis{\mathscr{P}}
\def\ee{\mathrm{e}}

\def\KKK{\mathscr{K}}
\def\sop{\text{stop}}

\def\Ce{\text{\rm C}}
\def\Cb{\Ce_{\text{\rm b}}}

\def\co{\text{\rm c}_{\text{0}}}
\def\taum{\tau_{\text{\rm m}}}
\def\tauc{\tau_{\text{\rm c}}}

\def\badj{^\circ}
\def\dd{\:\mathrm{d}}

\def\Ell{\mathrm{L}}
\def\MM{\text{\rm M}}

\def\lm{\rightarrow}
\def\NN{\mathbb{N}}
\def\defeq{\mathrel{:=}}
\def\rg{\mathop{\mathrm{rg}}}

\def\LL{\mathscr{L}}
\def\RR{\mathbb{R}}

\def\ordplus{{\stackrel{.}{+}}}

\theoremstyle{plain}
\newtheorem{theorem}{Theorem}[section]
\newtheorem{proposition}[theorem]{Proposition}
\newtheorem{lemma}[theorem]{Lemma}%
\theoremstyle{definition}
\newtheorem*{remark*}{Remark}
\newtheorem{example}[theorem]{Example}

\newtheorem{hypothesis}{Hypothesis}

\newtheorem{definition}[theorem]{Definition}

\begin{document}
\title{Adjoint bi-continuous semigroups and semigroups on the space of measures}
\author{\|B\'alint |Farkas|, Darmstadt}

\begin{abstract}
For a given bi-continuous semigroup $(T(t))_{t\geq 0}$ on a Banach space $X$ we define its adjoint on an appropriate closed subspace $X\badj$ of the norm dual $X'$. Under some abstract conditions this adjoint semigroup is again bi-continuous with respect to the weak topology $\sigma(X\badj,X)$. An application is the following: For $\Omega$ a Polish space we consider operator semigroups on the space $\Cb(\Omega)$ of bounded, continuous functions (endowed with the compact-open topology) and on the space $\MM(\Omega)$ of boun\-ded Baire measures (endowed with the weak$^*$-topology). We show that bi-continuous semigroups on $\MM(\Omega)$ are precisely those that are adjoints of a bi-continuous semigroups on $\Cb(\Omega)$. We also prove that the class of bi-continuous semigroups on $\Cb(\Omega)$ with respect to the compact-open topology coincides with the class of equi\-con\-ti\-nuous semigroups with respect to the strict topology. In general, if $\Omega$ is not Polish space this is not the case.
\end{abstract}

%\address{Technische Universit\"at Darmstadt, Fachbereich Mathematik\newline
%Schlossgartenstr. 7, D-64289 Darmstadt (Germany)}
%\email{\tt farkas@mathematik.tu-darmstadt.de}
%
%

\keywords
Not strongly continuous semigroups,  bi-continuous semigroups; adjoint semigroup; mixed-topology, strict topology; one-parameter semigroups on the space of measures
\endkeywords

\subjclass
47D06, 47D03, 47D99; 46A03
\endsubjclass

\maketitle

\section{Introduction}%
\label{s:intro}

The probably simplest example of a semigroup on the space $\Cb(\RR)$, namely the shift semigroup, fails to be strongly continuous and even measurable with respect to the sup-norm. To overcome this and the probably the more important fact, namely the failure of strong continuity of many transition semigroups, several different approaches have been developed, such as those of $\pi$-semigroups by Priola \cite{priola:1999} or  weakly continuous semigroups by Cerrai \cite{cerrai}. One could even simply consider other locally convex topologies on $\Cb(\RR)$ than the sup-norm-topology as, e.g., done by Dorroh, Neuberger \cite{dorroh/neuberger:1993}. A more recent abstract approach is that of \emph{integrable semigroups on norming dual pairs} due to Kunze, see \cite{kunze,kunze2} and \cite{kunze:arx2}). In this paper, we give preference to the notion of \emph{bi-continuous semigroups} initiated by K\"uhnemund \cite{Ku01,Ku03}.

\medskip The reason for giving distinction to this class of semigroups is the fact that there is already a vast abstract theory developed for them: there are generation, approximation  and perturbation results (see K\"uhnemund \cite{Ku01}, \cite{Ku03} and Farkas \cite{farkas-bdd,Fa04b}). Even the Hille--Phillips functional calculus was formulated in this setting, and was used to prove convergence rates for rational approximation schemes and for efficient Laplace inversion formulas, see Jara \cite{jara:2008}. More recently mean-ergodic theorems for bi-continuous semigroups have been studied by Albanese, Lorenzi, Manco \cite{ALM}. On the other hand, bi-continuous semigroups have appeared in several applications concerning parabolic and equations with unbounded coefficients on the space of bounded continuous functions. We mention here the papers: Albanese, Lorenzi, Manco \cite{ALM}, Albanese, Mangino \cite{AlbMan}, Es-Sarhir, Farkas \cite{EsFa1}, \cite{EsFa2}, Farkas, Lorenzi \cite{Luca}, Lorenzi, Zamboni \cite{LorenziZamboni}, Metafune, Pallara, Wacker \cite{MPW}. It is the aim of the present note to complement the existing theory of bi-continuous semigroups by the construction of the \emph{adjoint semigroup} (this is done in Section \ref{s:adj}).

\medskip
Beside these, recently semigroups on spaces of measures have been attracting much attention, see Manca \cite{manca:2008}, Lant, Thieme \cite{LT07}. It is not surprising, however, that such questions were addressed much earlier, for example by Sentilles \cite{sentilles:1970}, who studied operator semigroups of $\Cb(\Omega)$ (bounded continuous functions) and on $\MM(\Omega)$ (bounded measures), where $\Omega$ is a locally compact space. The important construction there was that of a \emph{strict-topology}, which will be also crucial in this paper (it is the topology considered by Dorroh, Neuberger in \cite{dorroh/neuberger:1993}). In this respect we will rely on the paper by Sentilles \cite{sentilles:1972}. In Section \ref{s:meas} below, the adjoint bi-continuous semigroup construction mentioned above will be applied to study bi-continuous semigroups on the space $\MM(\Omega)$ of bounded Baire measures. We further show that bi-continuous semigroups in some cases may be put in the theory of equi\-con\-ti\-nuous semigroups on locally convex spaces, but we also give an example, a rather pathological one, showing that this is not always possible (Section \ref{s:cnt}). Our result on the automatic equicontinuity of semigroups (Theorem \ref{thm:important}) has been obtained independently by M.~Kunze even in a more general situation, see \cite{kunze2}.

\medskip
Let us first we recall some terminology and set up the framework (see K\"uhnemund \cite{Ku01,Ku03}). For considering bi-continuous semigroups one needs a Banach space $(X,\|\cdot\|)$ which is endowed with an additional locally convex topology $\tau$. The two  topologies need somehow be connected, hence we assume the following:

\begin{hypothesis}\label{hyp:1}\hskip0em
\begin{enumerate}[(i)]
\item\label{enum:1:i} $\tau$ is Hausdorff and coarser than the norm-topology.
\item\label{enum:1:ii} The locally convex space $(X,\tau)$ is sequentially complete on $\tau$-closed, norm-boun\-ded sets.
\item\label{enum:1:iii} The dual space $(X,\tau)'$ is norming for $(X,\|\cdot\|)$, i.e.,
\begin{equation*}
\|x\|=\sup_{\varphi\in(X,\tau)' \atop \|\varphi\|\leq
1}|\varphi(x)|.
\end{equation*}
\end{enumerate}
\end{hypothesis}
\noindent Now, $\tau$-bi-continuous semigroups are defined as follows:
\begin{definition}\label{def:bicont}
A one-parameter semigroup $(T(t))_{t\geq 0}$ of bounded linear operators on a Banach space $X$ is called a ($\tau$-)\emph{bi-continuous
semigroup}, if
\begin{enumerate}[(i)]
\item the orbit $\RR_+\ni t\mapsto T(t)x$ is $\tau$-continuous for all $x\in X$ ($\tau$-strongly continuous),
\item $t\mapsto T(t)$ is a norm-bounded function, say, on $[0,1]$, in which case
it is exponentially bounded on $\RR_+$.
\item $(T(t))_{t\geq 0}$ is locally-bi-equi\-con\-ti\-nuous, which means
that for a norm-bounded $\tau$-null sequence $x_n$ the convergence
$T(t)x_n\to 0$ holds in the topology $\tau$ and uniformly in compact
intervals.
\end{enumerate}
\end{definition}

The main feature of this definition is that it mixes  \emph{properties with respect to two topologies}: the norm-topology of Banach spaces (thus it allows for norm-estimates) and a weaker notion of convergence. As in (iii) above, for functions $T:\RR_+\to\LL(X)$ we use the term ``\emph{locally ...}'',  if the property ``...'' is satisfied for operators $T(t)$, $t$ ranging over compact intervals of $\RR_+$.

Let us first present some examples for bi-continuous semigroups.
\begin{example}\label{ex:cb}
The illustrious examples  of bi-continuous semigroups are those on the space $\Cb(\Omega)$ when this space is endowed with the compact-open topology $\tauc$.  To be more specific, consider $\Omega$ a locally compact Hausdorff, or a metrisable topological space, N.B.~a completely regular space, (or even more generally a completely regular, $k_f$-space, i.e., a space $\Omega$ for which the continuity of a function $f:\Omega\to \RR$ is decided already on compact sets). The linear space of continuous and bounded functions $f:\Omega\to \RR$ becomes  a Banach space when if endowed with the supremum-norm $\|\cdot\|_\infty$. The additional topology that we consider is the \emph{compact-open topology $\tauc$} generated by the following family of semi-norms:
\begin{equation*}
\Bigl\{ p_K(f):=\sup_{x\in K}|f(x)|:\:K\subseteq \Omega\mbox{
compact}\Bigr\}.
\end{equation*}
It is trivial that for $X=\Cb(\Omega)$ and $\tau=\tauc$  Hypothesis \ref{hyp:1} is satisfied. Now some important examples of bi-continuous semigroups in this settings:
\begin{enumerate}[{\bf 1.}]
\item The shift semigroup on $\Cb(\RR)$ is bi-continuous for the compact-open topology.
\item If $\Omega$ is Polish space Dorroh and Neuberger  have studied semigroups $(T(t))_{t\geq 0}$ on $\Cb(\Omega)$ induced by jointly-continuous flows, see \cite{dorroh/neuberger:1993,dorroh/neuberger:1996}. K\"uhnemund \cite[Sec.~3.2]{Ku01} has shown that these semigroups are bi-continuous with respect to the topology $\tauc$.
\item Given $H$ a separable Hilbert space, Ornstein--Uhlenbeck semigroups on $\Cb(H)$ have proved to be bi-continuous semigroups with respect to the compact-open topology, see K\"uhnemund \cite[Sec.~3.3]{Ku01} (or \cite{Fa04b}).
\item Metafune, Pallara, Wacker proved  in \cite{MPW} that solutions of certain second order parabolic equations with unbounded coefficients give rise to $\tauc$-bi-continuous semigroups on $\Cb(\RR^d)$.
\end{enumerate}
There are many other instances of bi-continuous semigroups of this kind. Without completeness we mention the following  references: Albanese, Lorenzi, Manco \cite{ALM}, Lorenzi, Zamboni \cite{LorenziZamboni},  Es-Sarhir, Farkas \cite{EsFa1}, Farkas, Lorenzi \cite{Luca}.
\end{example}

The next example concerns the weak$^*$-topology on a dual of a Banach space (and includes, among others, the shift semigroup on $\Ell^\infty(\RR)$).

\begin{example}\label{e:KS}
Let $E$ be a Banach space $X=E'$ and $\tau=\sigma(E',E)$ the weak$^*$-topology. If $(T(t))_{t\geq 0}$ is strongly continuous semigroup on $E$ with respect to the norm, then its adjoint $(T'(t))_{t\geq 0}$ is a bi-continuous on $X$ with respect to the weak$^*$-topology (see \cite[Sec.~3.5]{Ku01}). It is a consequence of the Krein-\v Smulian Theorem (see \cite[Sec.~IV.6]{schaefer:1999}) and is shown in \cite{farkas-bdd} that if $E$ is separable, then every $\tau$-bi-continuous semigroup on $X$ is of this form. We show in Section \ref{s:cnt} that one cannot drop the separability assumption.
\end{example}

The last example illustrates that bi-continuous semigroups naturally appear in operator theory, too.
\begin{example}
Let $(T(t))_{t\geq0}$ and $(S(t))_{t\geq0}$ be $C_0$-semigroups (strongly continuous for the norm) on a Banach space $E$. Consider $X=\LL(E)$ endowed besides the operator norm also with the strong operator topology $\tau_{\sop}$.
It is obvious that Hypothesis \ref{hyp:1} is satisfied. The
\emph{implemented semigroup} $(U(t))_{t\geq0}$ is defined by
$$
U(t)\defeq L_{S(t)} R_{T(t)},
$$
where $L$ and $R$ stand for the left and the right multiplication by the indicated  linear operator, respectively.
It is easy to see that $(U(t))_{t\geq0}$ is a
semigroup on $X$ which is bi-continuous with respect to $\tau_{\sop}$ (see K\"uhnemund \cite[Sec. 3.4]{Ku01} and Alber \cite{Alber}).
\end{example}

\subsection*{The mixed-topology}
To close the introduction we present a construction for the so-called mixed topology, which allows us to handle the ``two-topologies feature'' of bi-continuous semigroups by means of a single locally convex topology. Let $(X,\tau)$ be as in Hypothesis \ref{hyp:1} and $\semis$ be a family of seminorms determining $\tau$ such that $p\leq \|\cdot\|$ for all $p\in \semis$, and $\|\cdot\|=\sup_{p\in \semis} p(x)$. For $(p_n)\subseteq \semis$ and $(a_n)\in \co$, $a_n\geq 0$ (positive null-sequence) consider the
seminorm
$$
\Tilde p_{(p_n,a_n)}(x):=\sup_{n\in\NN} a_n p_n(x).
$$
Let $\taum$ be the locally convex topology, called the \emph{mixed topology}, determined by the family of
seminorms
$$
\Tilde \semis\defeq\bigl\{\Tilde p_{(p_n,a_n)}:(p_n)\in\semis,
(a_n)\in \co,\:a_n\geq 0\bigr\}.
$$
\noindent  It is clear, that $\tau$ is coarser and the norm-topology is finer than $\taum$. Wiweger in \cite{wiweger:1961} presents a very general construction for the mixed-topology (without the assumptions in Hypothesis \ref{hyp:1} and even on not necessarily locally convex spaces). As a consequence of Examples D) and E) and in Theorem 3.1.1 of \cite{wiweger:1961} one obtains that $\taum$ is the finest locally convex topology on $X$ that coincides with $\tau$ on norm-bounded sets, for instance, if
\begin{enumerate}[1)]
\item $X=\Cb(\Omega)$ endowed with the sup-norm, $\Omega$ a completely regular space, and $\tau=\tauc$ the compact-open topology on $\Cb(\Omega)$ (cf.~Example \ref{ex:cb});
\item or, if $X=E'$, $E$ a Banach space and $\tau=\sigma(E',E)$ the weak$^*$-topology (cf.~Example \ref{e:KS}).
\end{enumerate}
By a routine argument one proves the following lemma (or see \cite[Theorem 2.3.1]{wiweger:1961}):
\begin{lemma}
\label{lem:mixedconv}
A sequence $x_n\in X$ is convergent in the topology $\taum$, if and only if it is norm-bounded and $\tau$ convergent.
\end{lemma}%
%\begin{proof} Note that it suffices to consider the case when the limit is $0$. First let $x_n\tmlm 0$ then clearly $x_n\tlm 0$. Suppose by contradiction that
%$\|x_n\|$ is unbounded. We may assume that $\|x_n\|\geq n$. For each
%$x_n$ take $p_n\in\semis$ such that $ p_n(x_n)\geq \tfrac{n}{2}. $
%Thus for $\Tilde p_{(p_k,1/k)}$ we have
%$$
%\Tilde p_{(p_k,1/k)}(x_n)=\sup_{k\in \NN\atop k>0}
%\tfrac{1}{k}p_k(x_n)\geq \tfrac{1}{n}p_n(x_n)\geq \tfrac{1}{2},
%$$
%which is a contradiction. \par\noindent
% Conversely, assume that  $x_n\tlm
%0$ and $x_n$ is norm-bounded. Take $a_n\in \co$ and $p_n\in\semis$.
%Then for $\varepsilon>0$ there exists $N\in\NN$ such that
%$|a_n|\leq\varepsilon$ whenever $n>N$. Thus
%$$
%\Tilde p_{(p_k,a_k)}(x_n)=\sup_{k\in\NN}a_k p_k(x_n)\leq\sup_{0\leq
%k\leq N}a_k p_k(x_n)+\sup_{k\geq N}a_k p_k(x_n)\leq C
%p'(x_n)+\varepsilon \|(a_k)\|,
%$$
%which shows $x_n\tmlm 0$.
%\end{proof}

\noindent The following result translates the notion of
bi-continuous semigroups to the language of mixed topologies.
\begin{proposition}\label{prop:mixsg}
The class of $\tau$-bi-continuous semigroups and the class of $\taum$\--strong\-ly
continuous and locally sequentially $\taum$-equi\-con\-ti\-nuous
semigroups coincide.
\begin{proof}
Let $(T(t))_{t\geq 0}$ be a $\tau$-bi-continuous semigroup. Since $[0,1]\ni
h\to T(h)x$ is norm-bounded and $\tau$-continuous for all $x$, we
obtain by Lemma \ref{lem:mixedconv} that these orbits are also
$\taum$-continuous. The  sequential $\taum$-equicontinuity of the
family
$$
\left\{T(t):t\in[0,t_0]\right\}
$$
is simply a reformulation of Definition \ref{def:bicont}.(iii) in view of
Lemma \ref{lem:mixedconv}.

\medskip For the converse, let
$(T(t))_{t\geq 0}$ be a $\taum$-strongly continuous and locally
se\-quen\-ti\-ally-equi\-con\-ti\-nu\-o\-us semigroup. Then the orbits
$[0,1]\ni h\to T(h)x$ are norm-bounded since they are
$\taum$-continuous. The $\tau$-strong continuity is immediate, while
the local bi-equicontinuity follows  again  by Lemma
\ref{lem:mixedconv}.
\end{proof}
\end{proposition}

\section{Adjoint of bi-continuous semigroups}\label{s:adj}
Given a bi-continuous semigroup $(T(t))_{t\geq 0}$ we would like to define its adjoint in such a way that it will be again a bi-continuous semigroup hence fitting in this theory. First, we have to specify the space with properties listed in Hypothesis \ref{hyp:1}. The next proposition gives a candidate.
\begin{proposition}\label{p:invari}
Given $X$ a Banach space with the additional topology $\tau$  satisfying Hypothesis \ref{hyp:1}, denote by $X\badj$ be the set of all norm-bounded linear
functionals which are $\tau$-sequentially continuous on norm-bounded sets of $X$. Then $X\badj$ is a closed linear subspace of the norm-dual $X'$, hence it is a Banach space.
\end{proposition}
\begin{proof}
That $X\badj$ is a linear subspace is trivial.
Take $\varphi_n\in
X\badj$ with $\|\varphi_n-\varphi\|\lm 0$, where $\varphi\in X'$. We
have to show that $\varphi\in X\badj$. To this end, consider a
norm-bounded $\tau$-null sequence $(x_n)$.  Then $\varphi(x_n)\lm 0$
follows from
\begin{equation*}
\begin{split}
|\varphi(x_n)|&\leq|\varphi(x_n)-\varphi_k(x_n)|+|\varphi_k(x_n)|\leq\\
&\leq K\|\varphi-\varphi_k\|+|\varphi_k(x-x_n)|\leq
\tfrac{\varepsilon}{2}+\tfrac{\varepsilon}{2}=\varepsilon,
\end{split}
\end{equation*}
first by taking $k\in\NN$ sufficiently large and then for fixed $k$
using the continuity assumptions on $\varphi_k$.
\end{proof}
 So the norm  inherited from $X'$ makes $X\badj$ a Banach space.  We equip $X\badj$  additionally  with the weak topology $\tau\badj\defeq\sigma(X\badj,X)$. It is our aim to consider bi-continuous semigroups with respect to this topology on $X\badj$. For this purpose we have to verify the validity of Hypothesis \ref{hyp:1}. Clearly,
$\tau\badj$ is Hausdorff since $X$ separates the points of $X\badj$, and also $X$  is norming.
Trivially, $\tau\badj$ is coarser than the norm-topology on $X\badj$. It remains to check only the $\tau\badj$-sequential completeness on
closed, norm-bounded sets, but this generally may fail to hold, so we incorporate this requirement into our hypotheses.
\begin{hypothesis}\label{hyp:2} Suppose that $X\badj\cap \overline{B(0,1)}$ is sequentially complete for $\sigma(X\badj,X)$.
\end{hypothesis}
The next example shows that this assumption is indeed restrictive, i.e.~it is not a consequence of the general framework of Hypothesis \ref{hyp:1}.
\begin{example}
Let $E$ be a non-reflexive Banach space $X=E'$ its norm dual and $\tau=\sigma(E',E)$ the weak$^*$ topology. Suppose also that $E'$ is separable, whence $\sigma(E'',E')$ is metrisable on bounded sets. Then $X\badj=(X,\tau)'=E$ and $\tau\badj=\sigma(E,E')$, the weak topology.  Take $y\in E''\setminus E$  arbitrary with $\|y\|_{E'}\leq 1$,  and take $y_n\in E$ with $\|y_n\|_{E}\leq 1$ converging to $y$ in the weak$^*$ topology $\sigma(E'',E')$ (such a sequence exists by Goldstine's Theorem). This shows that $\tau\badj=\sigma(E,E')$ (being the restriction of $\sigma(E'',E')$ to $E$) is not complete on $\overline{B_{X\badj}(0,1)}$. For a concrete example take $E=\co$, $E'=\ell^1$, $E''=\ell^\infty$, $y=\mathbf{1}$ the constant $1$ sequence and $y_n=$the sequence with first $n$ members $1$ the others $0$. So $X=\ell^1$ with $\tau=\sigma(\ell^1,\co)$, and $y_n\to y$ in $\ell^\infty$ for $\sigma(\ell^\infty,\ell^1)$.
\end{example}

The framework of Hypothesis \ref{hyp:1} is now established. are ready to define adjoint bi-continuous semigroups. The proof of the next proposition is straightforward.
\begin{proposition}
Let $B\in\LL(X)$ be a norm-bounded linear operator which is also
$\tau$-sequentially continuous on norm-bounded sets. Then the adjoint $B'\in \LL(X')$ leaves $X\badj$ invariant.
\end{proposition}
For a linear operator $B\in \LL(X)$ that is also $\tau$-sequentially continuous on norm bounded sets, denote by $B\badj$ the restriction of $B'\in \LL(X')$ to $X\badj$.

\medskip  Take now a bi-continuous semigroup $(T(t))_{t\geq 0}$ on $(X,\tau)$. Then the operators
$T(t)\badj$ obviously form a semigroup $(T\badj(t))_{t\geq 0}$, which is $\tau\badj$-strongly
continuous by definition. The exponential boundedness of $(T\badj(t))_{t\geq 0}$ is
trivial. To establish the local bi-equicontinuity, we assume the
following on the underlying space.
\begin{hypothesis}\label{hyp:3}
Every $(\varphi_n)\subset X\badj$ norm-bounded $\tau\badj$-null sequence is $\tau$ equi\-con\-ti\-nuous on norm bounded sets.
\end{hypothesis}

Under this and the previous hypotheses we have the following.
\begin{proposition}
Let $(T(t))_{t\geq 0}$ be a $\tau$-bi-continuous semigroup on $(X,\tau)$, and suppose that Hypotheses \ref{hyp:2} and \ref{hyp:3} are satisfied. Then the semigroup $(T\badj(t))_{t\geq 0}$ is a $\tau\badj$-bi-continuous semigroup on $X\badj$ (recall: by definition $\tau\badj=\sigma(X\badj,X)$).
\end{proposition}
\begin{proof}
By what is said in the paragraph preceding Hypothesis \ref{hyp:3} it remains to show the local $\tau\badj$-bi-equicontinuity of $(T\badj(t))_{t\geq 0}$. To this end take
a norm-bounded $\tau\badj$-null sequence and $x\in X$. For $t_0>0$
one has the $\tau$-compactness of $\{T(t)x:\:t\in[0,t_0]\}$, thus by
the equicontinuity of $\varphi_n$ we have
\begin{equation*}
[T\badj(t)\varphi_n](x)=\varphi_n(T(t)x)\lm 0
\end{equation*}
uniformly on $[0,t_0]$ (use here that for the equicontinuous family $\varphi_n$ pointwise convergence is same as uniform convergence on compact sets; actually what is needed here is the adaptation of Theorem III.4.5.~Schaefer \cite{schaefer:1999} to our situation). This means precisely the $\tau\badj$-bi-equicontinuity
of $(T\badj(t))_{t\geq 0}$, therefore if $(T(t))_{t\geq 0}$ is a bi-continuous semigroup.
\end{proof}

\section{Semigroups on the space continuous functions and of measures}\label{s:meas}
In this section, we would like to carry out the adjoint construction from the previous section in the particular case when $X=\Cb(\Omega)$, the space of bounded and continuous functions. For this purpose we first need to study the mixed-topology on $\Cb(\Omega)$, and recall some results from Sentilles \cite{sentilles:1970} and, in slightly modified form, from Farkas \cite{Fa04b}.

\medskip  Let $\Omega$ be a Polish space or a $\sigma$-compact locally compact Hausdorff space. Consider the mixed topology $\taum$ which is finest locally convex topology that coincides with $\tauc$ on sup-norm bounded sets of $\Cb(\Omega)$ (see the end of Section \ref{s:intro}). The dual of $(\Cb(\Omega),\taum)$ is the space $\MM(\Omega)$ of bounded Baire measures in $\Omega$. We briefly indicate a way to see this. By assumption $\Omega$ completely regular. The dual of $\Cb(\Omega)$ (as a Banach space) is isomorphic to the space $\MM(\beta \Omega)$ of all bounded, Baire measures on the Stone-\v Cech compactification  $\beta\Omega$ of $\Omega$, and the isomorphism is given by $\varphi(f)=\int_{\beta\Omega} f\dd\mu$. (One can represent a continuous linear function even by regular Borel measures, in this respect we refer to Knowles \cite{knowles:1967} and Ma\v r\'{\i}k \cite{marik:1957}.)

If $\Omega$ is Polish then it is $G_\delta$, and if $\Omega$ is locally compact it is open in $\beta\Omega$, see, e.g., Walker \cite[Ch.~1]{walker:1974}. Furthermore as $\Omega$ is Lindel\"of and $G_\delta$ in $\beta\Omega$, it is also a Baire set in there (a space with this property was called absolute Baire by Negrepontis \cite{negrepontis:1967}, see also Frol\'{\i}k \cite{frolik:1970}). Therefore it is possible to identify $\MM(\Omega)$ with a subspace of $\MM(\beta\Omega)$ in the following way:
\begin{equation*}
\begin{split}
&\iota:\MM(\Omega)\to \MM(\beta\Omega),\\
&[\iota(\nu)](B)\defeq \nu(\Omega\cap B)\quad\mbox{ for all
$\nu\in\MM(\Omega)$ and $B\subseteq\beta\Omega$ Baire set.}
\end{split}
\end{equation*}
Then $\iota$ is an injection with
$$
\rg\iota=\left\{\mu:\mu\in\MM(\beta\Omega),\:|\mu|(\beta\Omega\setminus\Omega)=0\right\}.
$$
One can see that a measure $\mu\in\MM(\beta\Omega)$ gives rise to a linear functional $\varphi\in \Cb(\Omega)'$ which is not only norm-continuous, but also $\tauc$-continuous on norm-bounded sets, if and only if it  belongs to $\rg\iota$, i.e.~using the above identification, it is a Baire measure on $\Omega$ (see, e.g., \cite{Fa04b} or Sentilles \cite{sentilles:1972}).

The mixed topology, also called the \emph{strict topology} and denoted by $\beta_0:=\taum$ in this setting, has the next remarkable properties.
\begin{theorem}[Sentilles \cite{sentilles:1972}]\label{thm:sent1}
Let $\Omega$  either be a $\sigma$-compact, locally compact space, or a Polish space. Then the following assertions are true:
\begin{enumerate}[a)]
\item  $\beta_0=\mu(\Cb(\Omega),\MM(\Omega))$, the Mackey topology, where $\MM(\Omega)$ denotes space of all bounded Baire-measures on $\Omega$.
\item  A linear operator $T:\Cb(\Omega)\to \Cb(\Omega)$ is $\beta_0$-continuous, if and only if it is $\beta_0$-sequentially continuous. The same holds for linear functionals.
\item  Every $(\varphi_n)\subset (\Cb(\Omega),\beta_0)'$ $\sigma(\MM(\Omega),\Cb(\Omega))$-null sequence is $\beta_0$-equi\-con\-ti\-nuous.
\end{enumerate}
\end{theorem}
As a consequence of b) we have $\Cb(\Omega)\badj=\MM(\Omega)$. We now turn our attention to $\tauc$-bi-continuous semigroups.
\subsection*{Bi-continuous semigroups on $\Cb(\Omega)$}
By Proposition \ref{p:invari}, for a linear operator $B\in \LL(\Cb(\Omega))$ which is $\tauc$-continuous on norm bounded sets the (Banach space) adjoint $B'$ leaves $\MM(\Omega)$ invariant, its restriction is denoted by $B\badj$. The next lemma is proved in \cite{Fa04b}.
\begin{lemma}\label{lem:3.2}
Let $\Omega$ be a Polish space, and let $T:\RR_+\to\LL(\Cb(\Omega))$ be a $\tauc$-strongly continuous function consisting of operators that are $\tauc$-continuous on norm-bounded sets. For a norm-bo\-un\-ded, weak$^*$-compact set $\KKK\subseteq\MM(\Omega)$  and $t_0>0$ the set of measures
$$
\left\{T\badj(t)\nu:t\in[0,t_0],\:\nu\in\KKK\right\}
$$
is tight.
\end{lemma}

The next proposition  is also taken from \cite{Fa04b}. We repeat it here with a slight modification and the additional assertion concerning $\beta_0$-continuity.
\begin{proposition}\label{p:cbsg2}
Let  $T:\RR_+\to \LL(\Cb(\Omega))$ be $\tauc$-strongly continuous and locally norm-bounded.
 Suppose that $T(t)$ takes norm-bounded $\tauc$-null sequences into $\tauc$-null sequences. Then for all compact sets $K\subseteq \Omega$ and $\varepsilon>0$, there exists $M>0$ and $K'\subseteq \Omega$ compact such that
 $$
\sup_{x\in K}\bigl|(T(t)f)(x)\bigr|\leq M\sup_{x\in K'} \bigl|f(x)\bigr|+\varepsilon\|f\|_\infty\quad\mbox{}
 $$
 holds uniformly for $t$ in compact intervals of $\RR_+$.
In particular, it is locally-$\beta_0$-equi\-con\-ti\-nuous, or, which the same, it is $\tauc$-bi-equi\-con\-ti\-nuous.
\end{proposition}
\begin{proof} Let $\varepsilon>0$, $t_0>0$ and $K\subseteq \Omega$ be a compact set. Take a compact set $K'\subseteq \Omega$ such that
$ |T\badj(t)\delta_x|(\Omega\setminus K')\leq\varepsilon $ for all $t\in[0,t_0]$ and $x\in K$. Such
compact set exists by Lemma \ref{lem:3.2}. We then obtain
\begin{equation*}
\begin{split}
& \sup_{x\in K}|T(t)f(x)|=\sup_{x\in K}\biggl|\int_\Omega f\dd  T\badj(t)\delta_x\biggr|\leq \sup_{x\in K}\int_{K'} \hskip-2pt\left|f\right|\dd |T\badj(t)\delta_x|+\sup_{x\in K}\hskip-0.8em \int_{\Omega\setminus K'} \hskip-0.8em\left|f\right|\dd |T\badj(t)\delta_x|\leq\\
&\quad\leq \sup_{t\in [0,t_0]}\|T(t)\|\cdot\sup_{x\in
K'}|f(x)|+\varepsilon\|f\|,
\end{split}
\end{equation*}
which is the assertion.
\end{proof}
A variant of this result has been obtained independently by M.~Kunze, see \cite[Theorem 4.4]{kunze2}.
\noindent Let us summarise the above.

\begin{theorem}\label{thm:important}
Let $\Omega$ be a Polish space and a consider the set $\mathscr{F}$ of  $\tauc$-strongly continuous semigroups $(T(t))_{t\geq 0}$ on $\Cb(\Omega)$  for which each $T(t)$ is $\tauc$-sequentially continuous on sup-norm bounded sets. Then the class of $\tauc$-bi continuous semigroups and the class of $\beta_0$-locally equi\-con\-ti\-nuous, $\beta_0$-strongly continuous  semigroups both coincide with $\mathscr{F}$.
\end{theorem}

\subsection*{Bi-continuous semigroups on $\MM(\Omega)$}
We now study the adjoint of a $\tauc$-bi-continuous semigroup. To do that we have to verify Hypotheses \ref{hyp:2} and \ref{hyp:3}. The validity of Hypothesis \ref{hyp:2}, i.e., that  $\MM(\Omega)$ is $\sigma(\MM(\Omega),\Cb(\Omega))$-sequentially
complete if $\Omega$ is a Polish space, was already known to Alexandroff \cite{aleksandrov:1943}. Hypothesis \ref{hyp:3} is satisfied by Theorem  \ref{thm:sent1}.c). Thus the abstract results from Section \ref{s:adj} yield the following (recall: we abbreviate $\tau\badj:=\sigma(\Cb(\Omega),\MM(\Omega)$).

\begin{theorem}
Let $\Omega$ be a Polish space and $(T(t))_{t\geq 0}$ a $\tauc$-bi-continuous semigroup on $\Cb(\Omega)$. Then the semigroup $(T(t)\badj)_{t\geq 0}$ defined as $T(t)\badj:=T(t)'|_{\MM(\Omega)}$ is a $\tau\badj$-bi-continuous semigroup on the space of bounded Baire measures $\MM(\Omega)$ .
\end{theorem}
It is little surprising that the converse of this statement is also true:
\begin{theorem}
Let $\Omega$ be a Polish space. Let $(S(t))_{t\geq 0}$ be $\tau\badj$-bi-continuous semigroup on the space $\MM(\Omega)$. Then there is a $\tauc$-bi-continuous semigroup $(T(t))_{t\geq 0}$ on $\Cb(\Omega)$ with $T\badj(t)= S(t)$.
\end{theorem}
\begin{proof} For $f\in \Cb(\Omega)$ we set $$(T(t)f)(x):=\int_\Omega f\dd (S(t)\delta_x),$$
where $\delta_x$ denotes the Dirac-measure at the point $x\in \Omega$. We then have  $$\sup |(T(t)f)(x)|\leq \|S(t)\|\cdot \|f\|_\infty,$$ so $T(t)f$ is a bounded function. If $x_n\to x$ in $\Omega$ for $n\to \infty$, then $\delta_{x_n}\to \delta_x$ in $\MM(\Omega)$ with respect to $\tau\badj$. Since $S(t)$ is $\tau\badj$-continuous on $\MM(\Omega)$, we have $S(t)\delta_{x_n}\to S(t)\delta_x$ and  whence the continuity of  $T(t)f$ follows. Altogether we obtain that $T(t)\in \LL(\Cb(\Omega))$. Obviously $(T(t))_{t\geq 0}$ is an exponentially bounded semigroup. We have to show that for each fixed $t\geq 0$ the operator is $\tauc$-bi-continuous, and that the semigroup $(T(t))_{t \geq 0}$ is $\tauc$-strongly continuous. Then by Theorem  \ref{thm:important}  $(T(t))_{t\geq 0}$ is a $\tauc$-bi-continuous semigroup, and by construction $T(t)\badj=S(t)$ holds.

\medskip\noindent We first prove that for $t>0$ fixed $T(t)$ is $\tauc$-bi-continuous. Assume the contrary, i.e., that there is a sup-norm bounded sequence $f_n\in \Cb(\Omega)$ $\tauc$-convergent to $0$ (i.e.~$\beta_0$-convergent to $0$), a compact set $K\subseteq \Omega$ and $\varepsilon>0$ such that
$$
\sup_{x\in K}\bigl|(T(t)f_n)(x)\bigr|>\varepsilon\quad\mbox{for all $n\in \NN$}.
$$
For each $n\in \NN$ take a point $x_n\in K$ with
 $$
\bigl |(T(t)f_n)(x_n)\bigr|>\varepsilon.
 $$We can suppose by compactness that $x_n\to x$ for some $x\in K$. Then $\delta_{x_n}\to \delta_x$ in $\tau\badj$, and we obtain that $S(t)\delta_{x_n}$ is a $\tau\badj$-convergent sequence. By Theorem \ref{thm:sent1}.c) this sequence is $\beta_0$-equicontinuous. So by Schaefer \cite[Sec.~III.4.5]{schaefer:1999} we can deduce
$$
\bigl|(T(t)f_n)(x_n)\bigr|\leq \sup_{m\in \NN}\Bigl|\int_{\Omega}f_n\dd S(t)\delta_{x_m}\Bigr|\to 0\quad\mbox{for $n\to \infty$}.
$$
This is a contradiction.

\medskip\noindent To see the $\tauc$-strong continuity let $K\subseteq \Omega$ be a compact set. Assume by contradiction that there is $f\in \Cb(\Omega)$, $\varepsilon>0$  and $t_n\in [0,1]$ with $t_n\to 0$ for $n\to\infty$ such that
$$
\sup_{x\in K}\bigl|(T(t_n)f)(x)-f(x)\bigr|>\varepsilon\quad\mbox{for all $n\in \NN$}.
$$
For each $n\in \NN$ take point $x_n\in K$ with
$$
\Bigl|\int_\Omega f \dd (S(t_n)\delta_{x_n}-\delta_{x_n})\Bigr|=|(T(t_n)f)(x_n)-f(x_n)|>\varepsilon.
$$
By compactness we can pass to a subsequence and assume that $x_n$ converges to some $x\in K$. This means $\delta_{x_n}$ $\tau\badj$-converges to $\delta_x$. By the local $\tau\badj$-bi-equicontinuity of $(S(t))_{t\geq 0}$,  we must have
$$
\sup_{s\in[0,1]}\Bigl|\int_\Omega f \dd (S(s)\delta_{x_n}-\delta_{x_n})\Bigr|\to 0\quad\mbox{for $n\to\infty$},
$$
a contradiction.
\end{proof}

\section{Counterexamples}
\label{s:cnt}
A surprising fact is that though $\tauc$  is generally not metrisable, the continuity of norm-bounded linear operators on norm-bounded sets can be described by convergent sequences. It is clear that some kind of countability plays an important role here (cf.~metric or $\sigma$-compact spaces). Indeed, the simplest non-countable space gives rise to a counterexample to Theorem \ref{thm:important}, when $\Omega$ is not a Polish space.  More specifically we construct below a bi-continuous semigroup which is not is  $\beta_0$-locally-equi\-con\-ti\-nuous. For other illuminating, related examples we refer to Kunze \cite[Sec.~3]{kunze2}.

\begin{example}\label{ex:nonl}
Let $\Omega=\omega_1$ the first uncountable ordinal number and
$\upsilon$ be the order topology. Suppose
that $f_n\to 0$ in the topology $\tauc$. We claim that there exists
$\alpha\in\omega_1$ such that $f_n\lm 0$ uniformly on
$[\alpha,\omega_1)$.  Suppose the contrary, i.e., for all
$\alpha<\omega_1$ there exists $k\in\NN$, $k>0$ such that for all
$N\in\NN$ there exists $n\geq N$ and $x\in[\alpha,\omega_1)$ with
$|f_n(x)|>1/k$. For all $\alpha\in\omega_1$ we have $k_\alpha\in\NN$
and  we may assume that $k_{\alpha_\xi}=k$ for a cofinal sequence
$\alpha_\xi\in\omega_1$. By induction we choose a sequence
\begin{equation*}
x_{\alpha_{\xi_1}}< x_{\alpha_{\xi_1}}\ordplus 1<
x_{\alpha_{\xi_2}}<x_{\alpha_{\xi_2}}\ordplus
1<\cdots<x_{\alpha_{\xi_j}}<x_{\alpha_{\xi_j}}\ordplus1<\cdots
\end{equation*}
with $f_{n_j}(x_{\alpha_{\xi_j}})>1/k$. Since
$K\defeq\{\lim_{j\lm\infty}
x_{\alpha_{\xi_j}},\:x_{\alpha_{\xi_j}}:\:j\in\NN\}$ is compact and
\begin{equation*}
\sup_{y\in K}|f_{n_j}(y)|\geq\tfrac{1}{k}\quad\mbox{for all
$j\in\NN$},
\end{equation*}
we arrived to a contradiction. Thus we have the existence of
$\alpha\in\omega_1$ as asserted above. Now, consider the family
$\{[\xi,\omega_1):\:\xi>\alpha\}$, which has the finite intersection
property and thus by compactness possesses an accumulation point
$x\in\beta\Omega$. All $f_n$ extend to the Stone-\v Cech
compactification $\beta\Omega$ and $|f_n(y)|<\varepsilon$ for all
$y\in[\alpha,\omega_1)$ if $n\geq N$. Take $n\in \NN$, $n\geq N$. By
the  continuity of $f_n$ on $\beta\Omega$ we have a neighbourhood
$U$ of $x$ such that for all $y\in U$
\begin{equation*}
|f_n(y)-f_n(x)|\leq\varepsilon.
\end{equation*}
There exist $\xi\in(\alpha,\omega_1)$ with $\emptyset\neq U\cap
[\xi,\omega_1)\ni z$, so
\begin{equation*}
|f_n(x)|\leq |f_n(x)-f_n(z)|+|f_n(z)|\leq \varepsilon+\varepsilon.
\end{equation*}
Thus $f_n(x)\lm 0$, which shows that $\delta_x$ is
$\tauc$-sequentially-continuous (on norm-bounded sets). However, it
is clear that this is not $\tauc$-continuous on norm-bounded sets.

\medskip Consider now the $C_0$-semigroup $(T(t))_{t\geq 0}$ generated
by the bounded operator $A\defeq\mathbf{1}\otimes \delta_x$. Since $A$ is
idempotent the semigroup $T$ takes the form
\begin{equation*}
T(t)=\ee^{tA}=I-A+\ee^tA.
\end{equation*}
This semigroup is bi-continuous but none of $T(t)$, $t>0$ is $\tauc$-continuous on norm bounded sets.
\end{example}

We close this paper by the next counterexample complementing Example \ref{e:KS}.
\begin{example}
We present a $\sigma(E',E)$-bi-continuous semigroups on $X=E'$ that is not the adjoint of a strongly continuous semigroup on $E$ (where by Example \ref{e:KS} $E$ is a fortiori non-separable).
A Banach space $E$ said to have the \emph{Mazur property} if every weakly$^*$-sequentially continuous linear functional on $E$ is weakly$^*$ continuous. Not every Banach space has this property, for instance $E=\ell^\infty$ lacks it. (We refer to further details and examples, e.g., to Edgar \cite{Edgar}.)
Now let $E$ be a Banach space without the Mazur property and let $X:=E'$. Consider  a weakly$^*$-sequentially continuous functional $\varphi$ on $E'$ that is not weakly$^*$-continuous, and let $x\in E'$ be an element with $\varphi(x)=1$.  Set $A:=x\otimes \varphi$, which is obviously a bounded idempotent linear operator on $E'$. Now the semigroup with the asserted properties is given as in Example \ref{ex:nonl}: $T(t):=\ee^{tA}=I-A+\ee^{t}A$.
\end{example}

\section*{Acknowledgement}
The author is grateful to the anonymous referee for the careful reading,  constructive criticism and insightful remarks.
\def\cprime{$'$}
\providecommand{\bysame}{\leavevmode\hbox
to3em{\hrulefill}\thinspace}
\providecommand{\MR}{\relax\ifhmode\unskip\space\fi MR }
\providecommand{\MRhref}[2]{%
  \href{http://www.ams.org/mathscinet-getitem?mr=#1}{#2}
} \providecommand{\href}[2]{#2}

\parindent0pt

\medskip
{\em Author's address}:\\
{\em B\'alint Farkas}, Technische Universit\"at Darmstadt, Fachbereich Mathematik, Darmstadt, Germany, Schlossgartenstr. 7, D-64289.\\
 e-mail: \texttt{farkas@mathematik.tu-darmstadt.de}.

\end{document}